\pdfoutput=1
\documentclass{amsart}

\usepackage[breakable,skins]{tcolorbox}
\newtcolorbox{tbox}[1][]{%
    breakable,
    enhanced,
    colframe=blue,
    coltitle=white,
    #1
}

\usepackage{amssymb}
\usepackage{booktabs}
\usepackage{hyperref}
\usepackage{todonotes}
\usepackage{tikz-cd}
\hypersetup{
    colorlinks=true,       
    linkcolor=blue,          
    citecolor=blue,        
    filecolor=blue,      
    urlcolor=blue           
}
\usepackage[alphabetic,backrefs,lite]{amsrefs}
\usepackage{verbatim}

\usepackage[all]{xy}

\numberwithin{equation}{section}

\theoremstyle{plain}
\newtheorem{lemma}{Lemma}[section]
\newtheorem{theorem}[lemma]{Theorem}

\newtheorem{proposition}[lemma]{Proposition}
\newtheorem{corollary}[lemma]{Corollary}

\newtheorem{definition}[lemma]{Definition}

\newtheorem{remark}[lemma]{Remark}
\newtheorem{notation}[lemma]{Notation}

\DeclareMathOperator{\vdim}{vdim}
\DeclareMathOperator{\edim}{edim}

\newcommand{\pp}{\mathbb{P}}

\newcommand{\ls}{\mathcal{L}}

\newcommand{\p}{\mathbb{P}}
\newcommand{\Sec}{\mathbb{S}ec}
\newcommand{\Span}[1]{\langle#1\rangle}

\newcommand{\LL}{{\mathcal L}}
\newcommand{\PP}{{\mathbb P}}
\newcommand{\F}{{\mathbb F}}
\newcommand{\C}{\mathbb{C}}

\newcommand{\cdim}{\operatorname{codim}}

\newcommand{\Pic}{\operatorname{Pic}}

\newcommand{\Sym}{\operatorname{Sym}}

\renewcommand{\H}{\operatorname{H}}
\newcommand{\h}{\operatorname{h}}

\newcommand{\elis}[1]{\todo[inline,color=green!40]{#1}}

\begin{document}

\title{Waring Identifiability for powers of forms via degenerations}

\author{Alex Casarotti}
\address{Dipartimento di Matematica\newline Universit\`a degli Studi di Trento \newline via Sommarive 14
I-38123 \newline Povo di Trento (TN), Italy}
\email{alex.casarotti@unitn.it}

\author{Elisa Postinghel}
\address{Dipartimento di Matematica\newline Universit\`a degli Studi di Trento \newline via Sommarive 14
I-38123 \newline Povo di Trento (TN), Italy}
\email{elisa.postinghel@unitn.it}


\thanks{
Both authors are members of INdAM-GNSAGA}

\maketitle 

\begin{abstract}


We discuss an approach to the secant non-defectivity of the varieties parametrizing $k$-th powers of forms of degree $d$. It employs a Terracini type argument along with certain degeneration arguments, some of which are based on toric geometry. This implies a result on the identifiability of the Waring decompositions of general forms of degree kd as a sum of $k$-th powers of degree $d$ forms, for which an upper bound on the Waring rank was proposed by Fr\"oberg, Ottaviani and Shapiro.
 \end{abstract}

\section{Introduction}

Identifiability problems arise naturally in many fields of both applied and classical algebraic geometry. A variety $X \subset \p^N$ is said to be $h-$\textit{identifiable} if the general point of its $h-$secant variety has a unique decomposition as a sum of $h$ points of $X$.
A classical application of identifiability concerns particular polynomial decompositions.
The Waring problem for forms asks for a unique decomposition of a homogeneous polynomial $F_d \in \C[x_0,\dots,x_n]_d$ as a sum of $d-$th powers of linear forms, i.e. 
\begin{equation}\label{classical-Waring}
F_d=L_1^d+\dots+L_h^d,
\end{equation} 
with $L_i \in \C[x_0,\dots,x_n]_1$.
A necessary condition for identifiability is secant non-defectivity: a variety $X \subset \p^N$ of dimension $\dim(X)=n$ is said to be not $h-$(secant) defective if the $h-$secant variety $\Sec_h(X)$, defined as the Zariski closure of points in $\p^N$ lying in the span of $h$ points of $X$, has the expected dimension $\min\{N,h(n+1)-1\}$.
In \cite{COV} the authors proved that, for the Veronese case $X=\nu_d(\p^n)$ and for all subgeneric ranks $h$ (i.e. such that $\Sec_h(X)\subseteq\
\PP^N$ does not fill up the space), a general form $F$ of rank $h$ is identifiable, with a few well known exceptions.
In the case of generic rank, the situation is almost the opposite: in \cite{GM} it is proved that all forms of generic rank are not identifiable with the following exceptions: $(n,d,h)=(1,2k-1,k),(3,3,5),(2,5,7).$

In \cite{FOS} the authors initiated the investigation of a generalization of the classical Waring problems for forms. In particular they showed that a general form $F_{kd} \in \C[x_0,\dots,x_n]_{dk}$ can be written as a sum of at most $k^n$ $k-$th powers of forms $G_i$'s of degree $d$
\begin{equation}\label{FOS-Waring}
F_{kd}=G_1^k+\dots+G_{h}^k,
\end{equation} 
and  that this bound is sharp, i.e. when $d$ is sufficiently large, $k^n$ computes the generic rank. The secant defectivity of the varieties parametrizing $k-$th powers of forms of degree $d$ remains an open problem in general.

In this paper we address both the secant defectivity and the identifiability problems for such Waring decompositions.
Denote with $V_{n,d}^k$ the variety parametrizing $k-$th powers of homogeneous degree $d$ forms in $n+1$ variables: 
 $$V_{n,d}^k:=\{[F^k]|F \in \C[x_0,\dots,x_n]_d\}.$$

 Our first main result about secant non-defectivity is the following.

\begin{theorem}\label{main-thm1}
 The variety $V_{n,d}^k$ is not $h$-defective if  $k\ge 3$ and  $h \leq \frac{1}{N+1} {N+k-3 \choose N}$, where $N={n+d \choose d}-1$.
\end{theorem}
Our second result is about identifiability. A bridge from non-defectivity to identifiability was built in \cite{CM} first and then generalised in the recent \cite{MM}: whenever $X$ is a sufficiently regular variety (with non-degenerate Gauss map), then if $X$ is not $h-$defective, then $X$ is $(h-1)-$identifiable. Using this and Theorem \ref{main-thm1} we obtain what follows.
\begin{theorem}\label{main-thm2}
A general form $F \in \C[x_0,\dots,x_n]_{dk}$ of rank $h$ with $k \geq 3$ is identifiable whenever $$h \leq \min\left\{\frac{1}{N+1} {N+k-3 \choose N}-1, \left\lfloor\frac{{{n+kd} \choose {n}}}{N+1} \right\rfloor-1\right\}$$
\end{theorem}

We remark that in \cite{Nen}, the author showed that the secant defectivity of $V^k_{n,d}$ can be bounded asymptotically, using a direct algebro-computational argument, to $k^n-d^n$. In Section \ref{section-bound} we show that, for $d \gg k$, our bound of Theorem \ref{main-thm1} extends the latter. 

In order to prove Theorem \ref{main-thm1}, we brought together a Terracini type argument and several different degeneration techniques.
By a classical application of Terracini's Lemma, non-defectivity problems for secant varieties translate into the study of particular linear systems of hypersurfaces of projective space with prescribed singularities. The first systematic study was used in the proof of the celebrated Alexander and Hirschowitz Theorem for the case of classical Waring problems \eqref{classical-Waring}, where secant varieties of Veronese embeddings of $\PP^n$ correspond to linear systems of hypersurfaces of $\PP^n$ with prescribed double points in general position.
In the setting of the generalized Waring problem, as in \eqref{FOS-Waring} for $k \geq 2$, a direct translation to linear systems of hypersurfaces with only double point singularities is not possible. In order to prove secant non-defectivity in this case it is necessary to impose a larger base locus to our linear systems. In particular, we will be interested in studying the dimensions of linear systems $\ls:=\mathcal{L}_{N,k}(V,2^h)$ of hyperurfaces of $\p^N$ of degree $k$ that are singular at $h$ general points and that contain the $d$-thVeronese embedding of $\PP^n$, $V\subset\PP^N$.
The study of such linear systems is carried out by combining two types of degenerations  introduced in \cite{Po} and in \cite{CDM} and \cite{Pos-SecDeg} respectively.
We start by degenerating the ambient space $\p^N$ to a scheme with two components and, in turn, the linear system $\ls$ to a fibered product of two linear systems, one on each component,  which are somewhat easier to deal with than the original one. In fact one of them consists of hypersurfaces containing a linear subspace and a collection of double points, the other one consists of hypersurfaces containing $V$ and just one fat point of relatively large multiplicity with support on $V$.
Now, in order to study the latter, we perform a toric degeneration of the Veronese $V$ to a union of $n$-dimensional linear spaces, which will have  the effect of reducing further the study of the limit linear system.

It is worth mentioning that the study of Waring type problems and identifiability of symmetric tensors has been implemented also in the applied fields, from chemistry, biology to algebraic statistics.
Recently in \cite{BCMO}, the problem of identifiability for $k-$th powers of forms was linked to the identifiability of centered Gaussian mixture models in applied statistics.

\subsection{Organization of the paper}
Section \ref{preliminaries} contains all definitions and our Terracini type result that translates non-defectivity of $V^k_{n,d}$ to the study of $\ls$, Proposition \ref{identifiability tool proposition}.
In Section \ref{section-degeneration} we explain in detail the degenerations techniques, both in the classical and in the toric setting.
 In Section \ref{section-linearsystems} we analyse two auxiliary linear systems arising from the degeneration of $\ls$, Proposition \ref{dim-doublepoints-on-Lambda} and Corollary \ref{corollary-toric}.
Section \ref{section-proof} is devoted to the proof of the main technical result, i.e. Theorem \ref{main-thm-linearsystems}.
Finally, in Section \ref{section-bound} we explain to what extent our bounds are asymptotically better then the ones known before in the literature.

\subsection{Acknowledgments}
The authors would like to thank Giorgio Ottaviani and Alessandro Oneto for several useful discussions during the preparation of this article.
We also thank the referee for their useful comments.

\section{Powers of forms}\label{preliminaries}
In order to give a coherent and self-contained treatment of the subject, in this section we recall some preliminary definitions and results.

We will  work over the field of complex numbers $\C$.

\subsection{Veronese embeddings}\label{preliminaries-Veronese}
Let $W:=\mathbb{C}^{n+1}$ and $W^\ast$ the dual vector space.
With $\p^n=\p(W)$ we denote the projective space over $\C$ of dimension $n$.
We introduce the following integers 
$$
N_d:={n+d \choose n}-1, \quad
N_d^{k}:={N_d+k \choose N_d}-1.
$$
 We will indicate $N_d$ simply by $N$, when no confusion may arise.

Notice that the following identities hold: 
\begin{align*}
\h^0(\p^n,\mathcal{O}_{\p^n}(d))&=N_d+1\\
\h^0(\p^{N_d},\mathcal{O}_{\p^{N_d}}(k))&=N_d^k+1
\end{align*}
where $\h^0(\p^a,\mathcal{O}_{\p^a}(b))$ denotes the number of global sections of the twisting sheaf $\mathcal{O}_{\p^a}(b)$ on $\PP^a$, for any integers $a\ge 1, b\ge0$. In other terms, ${{a+b}\choose b}$ is the dimension of the linear systems of hypersurfaces of degree $b$ of $\PP^a$ which, in turn, is the projectivization of the complex vector spaces of forms of degree $b$ in $a+1$ variables. 
%
With this in mind, we can make the following  identifications: $$\p^{N_d}=\p(\Sym^d(W^*)),\quad \p^{N_d^{k}}=\p(\Sym^k(\Sym^d(W^*))).$$

Now we consider the  following  \emph{Veronese embeddings}:
\begin{align*}
 \nu_d:& \ \p^n \longrightarrow V_{n}^d \subset \p(\Sym^d(W^*)) \\
 & \ [L] \longmapsto [L^d]\\
\intertext{and} 
 \nu_k:&\ \p^{N_d} \longrightarrow V_{N_d}^k \subset \p(\Sym^k(\Sym^d(W^*))) \\
& \ [F] \longmapsto  [F^k]
\end{align*}
where $L\in\C[x_0,\dots,x_n]_1$ is a linear form and $F\in\C[x_0,\dots,x_n]_d$  is a form of degree $d$.
The image of the embeddings are called \emph{Veronese varieties}. 

\begin{remark}

Note that both $\nu_d$ and $\nu_k$ are the maps corresponding to the complete linear systems associated with the line bundles $\mathcal{O}_{\p^n}(d)$ and $\mathcal{O}_{\p^{N_d}}(k)$ respectively. As elements of $ \p(\Sym^d(W^\ast))$ (respectively $\p(\Sym^k(\Sym^d(W^\ast)))$) the image of $\nu_d(p)$ (respectively $\nu_k(p)$), with $p$ a point, corresponds to the hyperplane parametrizing hypersurfaces of degree $d$ in $\p^n$ (respectively of degree $k$ in $\p^{N_d}$) passing through $p$.
\end{remark}

We want to parametrize forms in $\p^n$ of degree $dk$, i.e. elements in $\C[x_0,\dots,x_n]_{dk}$, that can be written as $k-$th powers of forms of degree $d$. 

\begin{remark}
Note that the Veronese varieties $\nu_{dk}(\p^n)$ are always contained in the set of all $k-$th powers of forms of degree $d$  because, trivially, $L^{dk}=(L^d)^k$.
\end{remark}

Now, we let
\begin{align*}
\phi_{dk}: & \ \p^{N_d} \longrightarrow \p^{N_{dk}}=\p(\Sym^{dk}(W^\ast))\\
& \  \ [F] \longmapsto [F^k]
\end{align*}
be the map that assigns to each form $F \in \C[x_0,\dots,x_n]_d$ its $k-$th power.

\begin{definition}\label{(d,k)-Veronese}
We call the scheme theoretic image $$V_{n,d}^k=\phi_{dk}({\p^{N_d}}) \subseteq \p^{N_{dk}}$$ the \textit{$(d,k)-$\emph{Veronese variety}}.
\end{definition}

Under the previous identification, the classical Veronese varieties correspond to the $(d,1)-$Veronese varieties.
On the other hand, for $k>1$, $V_{n,d}^k$ is not a standard Veronese variety, indeed it is easy to see that the target of $\phi_{dk}$ has dimension ${n+dk \choose dk}$, which is never equal to ${N_d+a \choose a}$ for any $a$. A priori we don't know if the map $\phi_{dk}$ is an isomorphism, as it happens for classical Veronese varieties, see Lemma \ref{immersion} below.

\subsection{Secant varieties and identifiability}
In this subsection we recall the definition of secant variety and the notion of identifiability, following  \cite{CM}.

Let $X\subset\p^N$ be a non degenerate reduced 
variety. Let $X^{(h)}$ be the $h$-th symmetric product of $X$,
that is the variety parameterizing unordered sets of $h$ points of $X$. 
Let $U_h^X\subset X^{(h)}$ be the smooth locus, given by sets of $h$
distinct smooth points.
\begin{definition}
  A point $z\in U^X_h$ represents a
set of $h$ distinct points, say $\{z_1,\ldots, z_h\}$. We say that a point $p\in \p^N$ is in the span of
 $z$, $p\in\langle z\rangle$,  if it is a linear combination of the
 $z_i$'s.
\end{definition}
With this in mind, one can define  the following object.


\begin{definition}\label{abstract} The \emph{abstract} $h$-\emph{secant variety} is the $(hn+h-1)$-dimensional variety
$$\textit{sec}_{h}(X):=\overline{\{(z,p)\in U_h^X\times\p^N| p\in \Span{z}\}}\subset
X^{(h)}\times\p^N.$$ 

Let $\pi:X^{(h)}\times\p^N\to\p^N$ be the projection onto the second factor. 
The $h$-\emph{secant variety} is
$$\Sec_{h}(X):=\pi(sec_{h}(X))\subset\p^N,$$
and $\pi_h^X:=\pi_{|sec_{h}(X)}:sec_{h}(X)\to\p^N$ is the $h$-\emph{secant map} of $X$.

If the variety $X$ is irreducible and reduced we  say that $X$ is
\textit{$h$-defective}
if $$\dim\Sec_{h}(X)<\min\{\dim\textit{sec}_{h}(X),N\}.$$
\end{definition}

The following is a classical result.

\begin{theorem}[Terracini's Lemma]
  \label{th:Terracini} Let $X\subset\p^N$ be an irreducible
  variety. Then the follwing holds.
  \begin{itemize}
 \item For any $p_1,\ldots,p_k\in X$ and $z\in\Span{p_1,\ldots,p_k}$, we have
$$\Span{T_{p_1}X,\ldots,T_{p_k}X}\subseteq T_z\Sec_k(X). $$
 \item There is a dense open set $U\subset X^{(k)}$ such that
$$\Span{T_{p_1}X,\ldots,T_{p_k}X}=T_z\Sec_k(X), $$
for any general point $z\in \Span{p_1,\ldots,p_k}$ with $(p_1,\ldots,p_k)\in U$.
  \end{itemize}
\end{theorem}

The following notions are related to the notion of secant variety.

\begin{definition}
  Let $X\subset\p^N$ be a non degenerate subvariety. We say that a
  point $z\in\p^N$ has \emph{rank $h$ with respect to $X$} if $z\in \langle p\rangle$, for some $p\in U_h^X$ and $z\not\in \langle p^\prime\rangle$ for any $p^\prime\in
  U_{h^\prime}^X$, with $h^\prime<h$.
\end{definition}


\begin{definition}\label{def:identifiability} A point $z\in\p^N$
  is $h$-\emph{identifiable with respect to} $X\subset\p^N$ if $z$ is of rank
  $h$ and $(\pi^X_h)^{-1}(z)$ is a single point. The variety $X$ is
  said to be \emph{$h$-identifiable} if the $h$-secant map $\pi_h^X$ is  birational, that
  is if the general point of $\Sec_h(X)$ is $h$-identifiable.
\end{definition}

It is clear, by Theorem \ref{th:Terracini}, that when $X$ is $h$-defective, or
more generally when $\pi_h^X$ is of fiber type, then  $X$ is not $h$-identifiable.

We now recall the recent result in \cite{MM}, in which the authors generalize the approach in \cite{CM} relating identifiability with the non defectivity of the secant variety. 

\begin{theorem}\label{Mass-Mella}
Let $X \subset \p^N$ be an irreducible and non-degenerate variety of dimension $n$, $h \geq 1$ an integer, and assume that:
\begin{itemize}
\item $(h+1)n+h \leq N$,
\item $X$ has non-degenerate Gauss map,
\item $X$ is not $(h+1)-$defective.
\end{itemize}
Then $X$ is $h-$identifiable.
\end{theorem}

In the next sections we will see how to rephrtase this theorem in the setting of powers of forms in order to give identifiability results for $k-$th powers of forms of degree $d$.

\subsection{Geometric construction of $(d,k)-$Veronese varieties}

Let us recall some facts from  \emph{apolarity theory}; the main reference is  \cite{Ger}.

\begin{notation}[Apolarity]	\label{apolarity-notation}
We consider two polynomial rings in $n+1$ variables, both endowed with the standard grading:
\begin{align*}
R&=\mathbb{C}[x_0,\dots,x_n]=\bigoplus_{i\in\mathbb{N}}R_i:=\mathbb{C}[x_0,\dots,x_n]_i\\
S&=\mathbb{C}[y_0,\dots,y_n]=\bigoplus_{i\in\mathbb{N}}S_i:=\mathbb{C}[y_0,\dots,y_n]_i.
\end{align*}
Interpreting the elements of $S$ as partial derivatives in the $x_i$'s, the pairing
$S_k\times R_l\to \mathbb{C}$, sends $(F_k,G_l)$ to the derivative $F_k\circ G_l\in R$ of $G_l$. 
If $k=l$ and if $\mathcal{I}\subset R$ is a homogeneous ideal, 
the \emph{orthogonal}  $\mathcal{I}_k^\perp\subset S_k$ is the following space of polynomials 
$$
\mathcal{I}_k^\perp=\{F\in S_k| F\circ G=0, \forall G\in \mathcal{I}_k\}.
$$
\end{notation}

It is a standard fact of representation theory for the linear group $GL(W)$, see for instance \cite{Lan}, that the space $\Sym^k(\Sym^d(W^\ast))$ can be decomposed as direct sum of $GL(W)-$modules in the following way: $$\Sym^k(\Sym^d(W^\ast))=\Sym^{dk}(W^\ast) \oplus \mathcal{E}$$ where $$\mathcal{E}=\H^0(\PP^N, \mathcal{I}_{V_n^d}(k))$$ is the $k-$th homogeneous part of the ideal of forms that vanish on the Veronese variety $V_n^d=\nu_d(\p^n)$, cf. notation of Section \ref{preliminaries-Veronese}. In slightly different terms, we can describe $\mathcal{E}=\H^0(\PP^N, \mathcal{O}_{\PP^N}(k)\otimes\mathcal{I}_{V_n^d})$ as the space of global sections of the linear system of degree$-k$ hypersurfaces of $\PP^N$ containing $V_n^d$.
We have the following exact sequence:
$$0 \rightarrow \mathcal{I}_{V_n^d}(k) \rightarrow \mathcal{O}_{\p^{N_d}}(k) \rightarrow \mathcal{O}_{V_n^d}(k) \rightarrow 0$$
Considering the long exact sequence in cohomology and using the fact that every Veronese variety is projectively normal, we have:
$$0 \rightarrow \H^0(\mathcal{I}_{V_n^d}(k)) \rightarrow \H^0(\mathcal{O}_{\p^{N_d}}(k)) \rightarrow \H^0(\mathcal{O}_{V_n^d}(k)) \rightarrow 0,$$
where we omit the indication of the underlying space in the expressions $H^0(\cdot)$, since it is clear from the context.
This shows that $$\dim(\mathcal{E})={N_d+k \choose k}-{n+dk \choose dk}.$$

It is easy to show that $GL(V)-$modules $\Sym^{dk}(W^\ast)$ and $\mathcal{E}$ are apolar, i.e. for every pair of forms $F,G \in \C[x_0,\dots,x_{N_d}]$ with $F \in \Sym^{dk}(W^\ast)$ and $G \in \mathcal{E}$, one has that $F \circ G=0$.

The constructions above fit into the following commutative diagram: 

\begin{displaymath}\label{diagram powers}
\xymatrix{
&&&\p(\Sym^k(\Sym^d(W^*))) \ar@{-->}[dd]^{\pi_{\mathbb{E}}} \\
 \p^n \ar[r] & \p(\Sym^d(W^*)) \ar[rru]_{\nu_k} \ar[rrd]_{\phi_{dk}}\\
&&&V_{n,d}^k \subset\p(\Sym^{dk}(W^*))
} 
\end{displaymath}
where the map $\pi_{\mathbb{E}}$ is the linear projection from the (projective) linear space $\mathbb{E}=\mathbb{P}(\mathcal{E})$.
We now prove that the map $\phi_{dk}$ is in fact an isomorphism.

\begin{lemma}\label{immersion}
With the above notations, we have $\Sec_2(V_{N_d}^k) \cap \mathbb{E}=\emptyset$. In particular the map $\phi_{dk}$ is an embedding.

\end{lemma}
\begin{proof}
Note that an element $F \in \Sec_2(V_{N_d}^k)$ is either of the form $F=L^k$, $F=L^{k-1}M$ or $F=M^k+N^k$, where $L,M,N \in \C[x_0,\dots,x_{N_d}]_1$ are linear forms in $\p(\Sym^k(\Sym^d(W^\ast)))$, or equivalently degree $d$ hypersurfaces in $\p^n$. In particular, if there exist such an $F$ with $F \in \mathcal{E}=\H^0(\mathcal{I}_{V_{n,d}}(k))$, then $V_{n,d}$ would be contained in the vanishing locus of $F$. Moreover the zero locus of $F$ is either a hyperplane ($\{L=0\}$), a union of two hyperplanes ($\{L=0\} \cup \{M=0\}$) or the union of a hyperplane and a subscheme $S$ of degree $k-1$ ($\{M+\zeta N=0\} \cup S$), where $\zeta$ is a $k-$th root of unity. Finally since $V_{n,d}$ is non-degenerate and irreducible the claim follows.

\end{proof}

\subsection{Identifiablity for $(d,k)-$Veronese varieties}

From now on we will work with the projective notation, in particular $\mathbb{E}$ has to be intended as the projectivization of the affine spave $\mathcal{E}$ introduced in the previous section.
Let us start by characterizing hyperplanes in $\mathbb{E}$ as particular linear subsystems of hypersurfaces in $\p^{N_d}$.
Denote by $\pi_{dk}$ the linear projection from the linear space $\p(\Sym^{dk}(W^\ast))$ to $\mathbb{E}$, i.e. $$\pi_{dk}:\p(\Sym^k(\Sym^d(W^\ast))) \dashrightarrow \mathbb{E}$$

\begin{lemma}\label{iperpiani E}
Let $\H^0(\mathcal{O}_{\p^{N_d^k}}(1) \otimes \mathcal{I}_{\p(\Sym^{dk}(W^\ast))})$ be the complete linear system of hyperplane sections of $\p(\Sym^k(\Sym^d(W^\ast)))$ containing the linear space $\p(\Sym^{dk}(W^\ast))$. 
Then 
$$\nu_k^*(\H^0(\mathcal{O}_{\p^{N_d^k}}(1) \otimes \mathcal{I}_{\p(\Sym^{dk}(W^\ast))})) \cong \H^0(\mathcal{O}_{\p^{N_d}}(k) \otimes \mathcal{I}_{V_{n,d}})$$
\end{lemma}

\begin{proof}
Let $H \in \H^0(\mathcal{O}_{\p^{N_d}}(k) \otimes \mathcal{I}_{V_{n,d}})$ be a degree $k$ hypersurface in $\p^{N_d}$ that contains the Veronese $V_{n,d}$. Then the linear span $\overline{H}$ of $(\nu_k)_{*}(H)$ is a hyperplane in $\p^{N_d^k}$ that contains $V_{n,dk}:=(\nu_k)_{*}(V_{n,d})$. Since $\Span{V_{n,dk}}=\p(\Sym^{dk}(W^\ast))$, we have that $$\H^0(\mathcal{O}_{\p^{N_d}}(k) \otimes \mathcal{I}_{V_{n,d}}) \subseteq \nu_k^*(\H^0(\mathcal{O}_{\p^{N_d^k}}(1) \otimes \mathcal{I}_{\p(\Sym^{dk}(W^\ast))})$$
To conclude, it is enough to observe that $$\h^0(\mathcal{O}_{\p^{N_d}}(k) \otimes \mathcal{I}_{V_{n,d}})=\dim(\mathcal{E})=\cdim(\p(\Sym^{dk}(W^\ast))),$$ 
where $\cdim$ here indicates the codimension in $\p(\Sym^k(\Sym^d(W^\ast)))$, and $\nu_k^*$ induces an isomorphism of global sections.
\end{proof}

We need the following easy technical lemma:

\begin{lemma}\label{finite F}

The linear projection $$\pi_{dk}: \p(\Sym^k(\Sym^d(W^\ast))) \dashrightarrow \mathbb{E}=\mathbb{P}(\H^0(\mathcal{I}_{V_{n,d}}(k)))$$ is generically finite when restricted to $V_{N_d,k}=\nu_k(\p(\Sym^d(W^\ast)))$.

\end{lemma}

\begin{proof}
The map $$\pi_{dk_{|\nu_k(\p(\Sym^d(W^\ast)))}}:V_{N_d,k} \dashrightarrow \mathbb{E}$$ is induced by a linear subsystem $\mathcal{F}$ of the line bundle $\mathcal{L}=\mathcal{O}_{\p^{N_d^k}}(1)$ 
with $$\nu_{k}^{*}(\mathcal{F})=|\mathcal{O}_{\Sym^d(W^\ast)}(1) \otimes \mathcal{I}_{V_{n,d}}(k)|.$$
Now the claim follows easily from the fact that $\H^0(\mathcal{I}_{V_{n,d}}(k))$ defines the Veronese variety $V_{n,d}$ set-theoretically.
\end{proof}

Let $p_1,\dots,p_h \in X \subset \p^N$ be general points. By Lemma \ref{th:Terracini}, we have that $\Sec_h(X)$ has the expected dimension if and only if $\H^0(\mathcal{O}_X(1) \otimes \mathcal{I}_{p_1^2,\dots,p_h^2})$ has the  \textit{expected} dimension, i.e. $$\dim \H^0(\mathcal{O}_X(1) \otimes \mathcal{I}_{p_1^2,\dots,p_h^2})=\max\{0,N+1-(n+1)h\}.$$

\begin{notation}
If $p_1,\dots,p_h \in X$ are general points, then we denote $\mathcal{L}_{h,X}:=|\mathcal{O}_X(1) \otimes \mathcal{I}_{p_1^2,\dots,p_h^2}|$.
\end{notation}

Before moving on to the explicit description of the identifiability for $V_{n,d}^k$, let us first prove a general proposition about linear systems of projected varieties.

\begin{proposition}\label{E,F}
Let $X \subset \p^N$ be a smooth non-degenerate projective variety. Moreover let $\p^N = \Span{\mathbb{F},\mathbb{E}}$ with $\mathbb{F},\mathbb{E}$ skew linear subspaces. Let $\pi_{\mathbb{E}}:\PP^N\to \mathbb{F}$ and $\pi_{\mathbb{F}}:\PP^N\to \mathbb{E}$ be the natural projections. If the projections restricted to $X$ are generically finite and $\mathcal{L}_{h,X}$ has the expected dimension, then $\mathcal{L}_{h,\pi_{\mathbb{F}}(X)}$ has the expected dimension if and only if $\mathcal{L}_{h,\pi_{\mathbb{E}}(X)}$ has the expected dimension.
\end{proposition}
\begin{proof}
Note that by symmetry of $\mathbb{E}$ and $\mathbb{F}$ it suffices to prove only one of the implications.
Let $q_1,\dots,q_h$ be general points on $\pi_{\mathbb{E}}(X)$, $x_i \in \pi_{\mathbb{E}}^{-1}(q_i)$ a general point in every fiber  and $z_i=\pi_{\mathbb{F}}(p_i)$.
By the generality assumption we have that $x_1,\dots,x_h$ are general and so are $z_1,\dots,z_h$. Since $\pi_{\mathbb{F}}$ restricted to $X$ is generically finite the space of hyperplanes $\mathcal{L}_{h,\pi_{\mathbb{F}}(X)}=|\mathcal{O}_{\mathbb{E}}(1) \otimes \mathcal{I}_{z_1^2,\dots,z_h^2}|$ correspond to $\mathcal{L}_{h,X}\otimes  \mathcal{I}_{\mathbb{F}}$. Now the splitting $\p^N= \Span{\mathbb{F},\mathbb{E}}$ induces the linear projection $$\pi:\mathcal{L}_{h,X} \mapsto \mathcal{L}_{h,X} \otimes \mathcal{I}_{\mathbb{E}}$$ such that $Ker(\pi)=\mathcal{L}_{h,\pi_{\mathbb{F}}(X)}$. Since, by assumption, both the source and the kernel have the expected dimension by the rank-nullity theorem the assertion follows.
\end{proof}

We are finally able to characterize the identifiability properties for the case of powers of forms. 
Consider the following linear system of all hypersurfaces of $\PP^{N_d}$ of degree $k$ containing the Veronese variety $V_{n,d}\subset\PP^{N_d}$ and double at the points $p_1,\dots,p_h$ that lie in general position in $\PP^{N_d}$:
$$\ls_{N_d}(V_{n,d},2^h):=\mathcal{O}_{\p^{N_d}}(k) \otimes \mathcal{I}_{V_{n,d}} \otimes \mathcal{I}_{p_1^2,\dots,p_h^2}.$$

\begin{proposition}\label{identifiability tool proposition}
In the above notation, let $p_1,\dots,p_h$ be general points in $\p^{N_d}=\p(\Sym^d(W^\ast))$. The linear system $\ls_{N_d}(V_{n,d},2^h)$
has the expected dimension if and only if $\Sec_h(V_{n,d}^k)$ has the expected dimension.

\end{proposition}

\begin{proof}

With the notations of Proposition \ref{E,F} we have $\mathbb{E}=\mathbb{P}(\H^0(\mathcal{I}_{V_{n,d}}(k)))$ and $\mathbb{F}=\Sym^{dk}(W^\ast)$.
We have that $\pi_{\mathbb{E}}$ restricted to $X=V_{N_d}^k$ is an isomorphism by Lemma \ref{immersion}, in particular it is generically finite. The same holds  for $\pi_{\mathbb{F}}=\pi_{dk}$ by Lemma \ref{finite F}. Now $\mathcal{L}_{h,V_{N_d}^k}$ has the expected dimension by Alexander-Hirschowitz (see Theorem \ref{AH-thm} below) and $$\mathcal{L}_{h,\pi_{dk}(V_{N_d}^k)}=|\mathcal{O}_{\p^{N_d}}(k) \otimes \mathcal{I}_{V_{n,d}} \otimes \mathcal{I}_{p_1^2,\dots,p_h^2}|$$ by Lemma \ref{iperpiani E}. 

\end{proof}

 
\section{Degeneration techniques}\label{section-degeneration}

In this section we will discuss two types of degenerations that will provide main tools for the proofs of the results of this article.

\subsection{The $\mathbb{F}\mathbb{P}-$degeneration}\label{degeneration-section}

In this section we recall a degeneration procedure introduced in \cite{Po}, which
consists in degenerating the projective space $\pp^N$ to a reducible variety with two components, and then studying degenerations of  line bundles on the general fiber. 

\subsubsection{Degenerating the ambient space}\label{degeneratingSpace-section}

Let $\Delta$ be a complex disc centred at the origin and consider the product $\mathcal{Y}=\pp^N \times \Delta$ with the natural projections $\pi^\mathcal{Y}_1:\mathcal{Y}\to \PP^N$ and $\pi^\mathcal{Y}_2:\mathcal{Y}\to \Delta$. The second projection  is a flat morphism and we denote by
$Y_t:=\pp^N \times \{t\}$ the fiber over $t\in\Delta$. We will refer to $Y_0$ and to $Y_t$, with $t\neq0$, as the \emph{central fiber} and the \emph{general fiber} respectively.
Let  $f:\mathcal{X}\to \mathcal{Y}$ denote the blow-up of $\mathcal{Y}$ at a point $(p,{0})\in Y_0$ in the central fiber. Consider the following diagram, where    $\pi^\mathcal{X}_i:= \pi^\mathcal{Y}_i\circ f$, for $i=1,2$:

\begin{displaymath}
\xymatrix{ 
\mathcal{X}  \ar[rr]_f \ar[drr]_{\pi^{\mathcal{X}}_2} \ar@/^/[rrrr]^{\pi^{\mathcal{X}}_1}  && \mathcal{Y} \ar[rr]_{\pi^{\mathcal{Y}}_1} \ar[d]^{\pi^{\mathcal{Y}}_2}   && \pp^N \\
&& \Delta }
\end{displaymath}
The morphism $\pi^{\mathcal{X}}_2:\mathcal{X} \rightarrow \Delta$ is flat with fibers denoted by $X_t$, $t \in\Delta$.  For the general fiber we have $X_t\cong Y_t=\pp^N$, while the central  fiber $X_0$ is the reduced union of  the strict transform of $Y_0$, that we shall denote with $\F$,  and the exceptional divisor $\pp\cong\pp^N$ of $f$. 
The two components $\pp$ and $\F$ meet transversally and we will denote  by $R$ the intersection: $R:=\F\cap \PP\cong\PP^{N-1}$.
We will say that $\PP^N$ \emph{degenerates} to $X_0=\pp\cup\F$.

We will now endow the general fiber $X_t$ with a line bundle and we will describe its limits on $X_0$ via this degeneration.
In order to do so, we will give bases for the Picard groups of the components of $X_0$. 

\begin{notation} 
We denote by $H_\pp$ the hyperplane class of $\pp$, so that the Picard group of the exceptional component is generated by $H_\pp$. Moreover we denote with $H_\F$ the hyperplane class of $\F$, pull-back of a general hyperplane of $Y_0\cong\PP^N$, and with $E:=\pp|_{\F}$ the exceptional class in $\F$: $H_\F$ and $E$ generate the Picard group of $\F$.
\end{notation}
In these bases, $R$ has class $H_\pp$ in $\textrm{N}^1(\pp)$ and  $E$ in $\textrm{N}^1(\F)$.
A line bundle  on $X_0$ will correspond to a line bundle on  $\pp$ and a line bundle on $\F$, which match on the intersection $R$. In other terms, we can describe the Picard group of $X_0$ as a fiber product  $$\textrm{Pic}(X_0)=\textrm{Pic}(\pp)\times_{\textrm{Pic}(R)}\textrm{Pic}(\F).$$ 
Consider the line bundle $\mathcal{O}_{\mathcal{X}}(k)=(\pi^\mathcal{X}_1)^{\ast}(\mathcal{O}_{\pp^N}(k))$ and the following twist by a negative multiple of the exceptional divisor: 
\begin{align*}
\mathcal{M}_\mathcal{X}(k,a)&:=\mathcal{O}_{\mathcal{X}}(k)\otimes\mathcal{O}_{\mathcal{X}}(-a\PP).
\intertext{The line bundle $\mathcal{M}_\mathcal{X}(k,a)$ will induce a line bundle on each fiber $X_t$ by restriction: }
\mathcal{M}_{t}(k,a)&:=\mathcal{M}_\mathcal{X}(k,a)|_{X_t}, \  t\in\Delta.
\intertext{For $t\ne 0$, since $\PP\cap X_t=\emptyset$,  we have }
\mathcal{M}_{t}(k,a)&=\mathcal{O}_{X_t}(k)\\
\intertext{while on the components of the central fiber we have}
\mathcal{M}_{\pp}(k,a)&:=\mathcal{M}_\mathcal{X}(k,a)|_{\pp}=\mathcal{O}_{\pp}(aH_{\pp}),\\ 
 \mathcal{M}_{\F}(k,a)&:=\mathcal{M}_\mathcal{X}(k,a)|_{\F}=\mathcal{O}_{\F}(kH_{\F}-aE).\end{align*}
The resulting line bundle on $X_0$ is a flat limit of the bundle $\mathcal{O}_{X_t}(k)\cong\mathcal{O}_{\pp^n}(k)$, for $t\to 0$.


\subsubsection{Degenerating the Veronese variety}\label{degeneratingVeronese-section}
We will use the same notation as in Section \ref{degeneratingSpace-section}.
Let us set $$N:=N_d={{n+d}\choose n}-1,$$ and let $V:=V_{n,d}=v_{d}(\pp^n)\subset\pp^N$ denote the $d$-th Veronese embedding of $\pp^n$ in $\pp^N$, and consider the $1$-parameter family $\mathcal{V} =V\times\Delta \subset \mathcal{Y}$ with the natural projections ${\pi^\mathcal{Y}_1}|_\mathcal{V}:\mathcal{V}\to \PP^N$ and ${\pi^\mathcal{Y}_2}|_\mathcal{V}:\mathcal{V}\to \Delta$. The second projection  is a flat morphism and we denote by
$V_t:=V \times \{t\}$ the fiber over $t\in\Delta$. 
We pick a general point $(p,0)\in V_0\subset Y_0$ in the central fiber of $\pi_2^\mathcal{Y}:\mathcal{Y}\to\Delta$  supported on the Veronese variety and we consider the blow-up $f:\mathcal{X}\to\mathcal{Y}$ at $(p,0)$. This induces the blow-up $f|_\mathcal{V}:\widetilde{\mathcal{V}}\to\mathcal{V}$ of $\mathcal{V}$ at $(p,0)$ and the fibers of 
$(\pi_2^{\mathcal{Y}}\circ f)|_{\mathcal{\widetilde{\mathcal{V}}}}$ are as follows: the general fiber is a Veronese variety 
\begin{align*}\widetilde{V}_t&\cong V,\\
\intertext{while the central fiber is the reduced union of two components,} 
\widetilde{V}_0&=\widetilde{V}_{\F}\cup \Lambda,\\
\intertext{where $\widetilde{V}_{\F}$ is the strict transform of $V_0$ under the blow-up at $p$, while $\Lambda\cong\pp^n$ is the exceptional divisor on $\widetilde{V}$. Moreover we can write} 
\Lambda_R&:=\widetilde{V}_{\F}\cap \Lambda\subset\Lambda,
\end{align*}
 and observe that $\Lambda_R\cong\pp^{n-1}$.

Consider the line bundle $\mathcal{M}_\mathcal{X}(k,a)$ and  twist it by the ideal sheaf of $\widetilde{\mathcal{V}}$:
\begin{align*}
\mathcal{M}_\mathcal{X}(k,a;\widetilde{\mathcal{V}})&:=\mathcal{M}_\mathcal{X}(k,a)\otimes\mathcal{I}_{\widetilde{\mathcal{V}}}=\mathcal{O}_{\mathcal{X}}(k)\otimes\mathcal{O}_{\mathcal{X}}(-a\PP)\otimes\mathcal{I}_{\widetilde{\mathcal{V}}}.\\
\intertext{%
This restricts to the following line bundles  on the fibers $X_t$:} 
\mathcal{M}_{t}(k,a;V)&=\mathcal{O}_{X_t}(k)\otimes\mathcal{I}_{V_t}, \  t\in\Delta\setminus\{0\},\\
\mathcal{M}_{\pp}(k,a;\Lambda)&
=\mathcal{O}_{\pp}(aH_{\pp})\otimes\mathcal{I}_{\Lambda},\\ 
 \mathcal{M}_{\F}(k,a;\widetilde{V})&
 =\mathcal{O}_{\F}(kH_{\F}-aE)\otimes\mathcal{I}_{\widetilde{V}}.\end{align*}
The resulting line bundle on $X_0$ is a flat limit of the bundle $\mathcal{O}_{\pp^n}(k)\otimes\mathcal{I}_{V}$ on the general fiber.


\subsubsection{Degenerating  a collection of points in general position}\label{degeneratingGeneralPoints-section}
We continue to use the notation of Sections \ref{degeneratingSpace-section}-\ref{degeneratingVeronese-section}. 
On the general fiber of $\mathcal{Y}\to\Delta$, i.e. for $t\ne 0$, we consider a collection of points $\{x_{1,t}\dots,x_{h,t}\}\subset Y_t$ in general position and that, in particular, lie off the Veronese variety $V_t\subset Y_t$. After choosing a point that lies generically on the Veronese variety in the central fiber,  $p\in V_0\subset Y_0$,  we degenerate each point $x_{1,t}\in Y_t$ to an infinitely near point to $p\in V_0$ as follows.
For every $i\in\{1,\dots,h\}$, consider  the curve $(\pi^\mathcal{Y}_1)^{-1}(x_i)\subset\mathcal{Y}$ and its pull-back  $C_i$ on $\mathcal{X}$. 
The union $\bigcup_i C_i$ intersects each fiber $X_t$ transversally in $h$ distinct points. For  $t\neq0$, the points $\bigcup_i C_i\cap X_t$ are in general position. Moreover, by the generality assumption, the curves  $(\pi^\mathcal{Y}_1)^{-1}(p_i)\subset\mathcal{Y}$ are not tangent to $V_0\in Y_0$, therefore the intersection points 
$C_i\cap X_0$ lie on $\PP$ but not inside $\Lambda$.


Consider the blow-up of  $\mathcal{X}$  along the union of curves $\bigcup_{i=1}^h C_t$, $g_h:\widetilde{\mathcal{X}}\to \mathcal{X}$, with exceptional divisors $E_{C_i}$. Since these curves are disjoint, the result does not depend of the order of blow-up.
The general fiber, that we continue to call $X_t$ by abuse of notation, is isomorphic to a $\PP^N$ blown-up at $h$ points in general position and its Picard group will be generated by the hyperplane class and by the classes of the exceptional divisors:
$$\Pic(X_t)=\mathbb{Z}\langle H,E_{1,t},\dots,E_{h,t}\rangle.$$
The central fiber  has two components, the pull-back of $\F$ and the strict transform of $\pp\cong\pp^N$, which is isomorphic to a $\PP^N$ blown-up at $h$ points  in general position: abusing notation, we call $\F$ and $\pp$ the two components, so that $X_0={\pp}\cup\F$. The Picard group of $\widetilde{\pp}$ is
$$\Pic({\pp})=\mathbb{Z}\langle H_\pp,E_{1,t},\dots,E_{h,t}\rangle.$$

For a vector ${\bf m}=(m_1,\dots,m_h)\in\mathbb{N}^n$, 
consider the following sheaf on $\widetilde{\mathcal{X}}$:
\begin{align*}
\mathcal{M}_{\widetilde{\mathcal{X}}}(k,a;\tilde{\mathcal{V}},{\bf m})&:=\mathcal{O}(k)\otimes\mathcal{O}(-a\PP)\otimes\mathcal{O}(-(m_1E_{C_1}+\cdots+m_hE_{C_h}))\otimes\mathcal{I}_{\widetilde{V}}.
\intertext{It restricts to}
\mathcal{M}_t(k,a;V,{\bf m})&=
\mathcal{O}(k)\otimes\mathcal{O}(-(m_1E_{1,t}+\cdots+m_hE_{h,t}))\otimes\mathcal{I}_{V_t},\ t\ne 0,\\
\intertext{on the general fiber, and to}
\mathcal{M}_\pp(k,a;\Lambda,{\bf m})&=
\mathcal{O}_{\pp}(aH_{\pp}-(m_1E_{1}+\cdots+m_hE_{h}))\otimes\mathcal{I}_{\Lambda},\\ 
\mathcal{M}_\F(k,a;\widetilde{V},{\bf m})&=
\mathcal{O}_{\F}(kH_{\F}-aE)\otimes\mathcal{I}_{\widetilde{V}}.
\end{align*}
on the components of the central fiber.
The resulting line bundle on $X_0$ is a flat limit of the bundle  on the general fiber. 
%

\subsubsection{Matching conditions}\label{matching conditions}

We will abbreviate the notation of the previous sections by setting 
\begin{align*}
\mathcal{M}_{\widetilde{\mathcal{X}}}&:=\mathcal{M}_{\widetilde{\mathcal{X}}}(k,a;\tilde{\mathcal{V}},{\bf m})
 \intertext{and} 
\mathcal{M}_t&:=\mathcal{M}_t(k,a;V,{\bf m})\\
\mathcal{M}_{\pp}&:=\mathcal{M}_\pp(k,a;\Lambda, {\bf m})\\
\mathcal{M}_{\F}&:=\mathcal{M}_\F(k,a;\tilde{V},{\bf m}).
\end{align*}
We are interested in computing the dimension of the space of global sections of the line bundle on the central fiber which, by semicontinuity, gives an upper bound  to the dimension of the space of global sections of the line bundle on the general fiber:
\begin{equation}\label{semicontinuity}
\h^0(X_0,\mathcal{M}_0)\ge \h^0(X_t,\mathcal{M}_t).
\end{equation}
In order to do so, we consider the natural restrictions to the intersection $R=\pp\cap\F$ of the central fiber:
\begin{align*}
0\to\hat{\mathcal{M}}_\pp&\to {\mathcal{M}}_\pp\to \mathcal{M}_\pp|_R\to0,\\
0\to\hat{\mathcal{M}}_\F&\to {\mathcal{M}}_\F\to \mathcal{M}_\F|_R\to0,
\end{align*}
where $\hat{\mathcal{M}}_\pp=\hat{\mathcal{M}}_\pp(k,a;\Lambda,{\bf m})$ and $\hat{\mathcal{M}}_\F=\hat{\mathcal{M}}_\F(k,a;\tilde{V},{\bf m})$ denote the kernels of the restriction maps. Since $R=H_\pp$ on $\pp$ and $R=E$ on $\F$, we have
\begin{align*}
\hat{\mathcal{M}}_\pp&=\mathcal{O}_{\pp}((a-1)H_{\pp}-(m_1E_{1}+\cdots+m_hE_{h})\otimes\mathcal{I}_{\Lambda}\\\
\hat{\mathcal{M}}_\F&=\mathcal{O}_{\F}(kH_{\F}-(a+1)E)\otimes\mathcal{I}_{\tilde{V}}.
\end{align*}
Consider the restriction maps of global sections: 
\begin{align*}
r_\pp:\H^0(\pp,\mathcal{M}_\pp)&\to \H^0(R,\mathcal{M}_\pp|_R),\\
r_\F:\H^0(\pp,\mathcal{M}_\F)&\to \H^0(R,\mathcal{M}_\F|_R).
\end{align*}
We notice that the spaces of global sections of the restricted systems are both subspaces of the space of global sections of the degree-$a$ line bundle on $R\cong\pp^{N-1}$:
$$
\H^0(R,\mathcal{M}_\pp|_R), \H^0(R,\mathcal{M}_\F|_R)\subseteq \H^0(R,\mathcal{O}_R(a)).
$$
A global sections of  $\mathcal{M}_0$ consists of an element of $\H^0(\pp,\mathcal{M}_\pp)$ and an element of $\H^0(\F,\mathcal{M}_\F)$ which match in $\H^0(R,\mathcal{O}_R(a))$, i.e. the space of global sections   $H^0(X_0,\mathcal{M}_0)$ is described as a fiber product via
the following commutative diagram:
$$
\begin{tikzcd}
  \H^0(X_0,\mathcal{M}_0) \arrow[r] \arrow[d]
    & \H^0(\pp,\mathcal{M}_{\pp}) \arrow[d,"r_\pp"] \\
  \H^0(\F,\mathcal{M}_{\F}) \arrow[r,"r_\F"]
&  \H^0(R,\mathcal{M}_\pp|_R\cap \mathcal{M}_\F|_R)
\end{tikzcd}
$$
This yields the formula for the dimension of the spaces of global sections
$$
\h^0(X_0,\mathcal{M}_{0})=\h^0(\pp,\hat{\mathcal{M}}_{\pp})+
\h^0(\F,\hat{\mathcal{M}}_{\F})+\h^0(R,\mathcal{M}_\pp|_R\cap \mathcal{M}_\F|_R),
$$ which, in terms of dimensions of line bundles, reads as follows:
\begin{equation}\label{H0 central fiber}
\dim\mathcal{M}_{0}=\dim\hat{\mathcal{M}}_{\pp}+\dim\hat{\mathcal{M}}_{\F}+\dim \mathcal{M}_\pp|_R\cap \mathcal{M}_\F|_R+2.
\end{equation}

\begin{remark}\label{abuse-notation-ls}
There is an obvious isomorpshism between  $\mathcal{M}_t(k,a;V,{\bf m})$, line bundle  on $X_t$, and the line bundle $\ls_{N,k}(V,{\bf{m}}):=\mathcal{O}_{\PP^N}(k)\otimes\mathcal{I}_{V}\otimes\mathcal{I}_Z$ on $\PP^N$, where $Z$ is a union of fat points in general position in $\PP^N$  with multiplicities respectively $m_1,\dots,m_h$. Such isomorphism is  given by taking strict transforms of elements of $\ls_{N,k}(V,{\bf{m}})$. 

Similarly, since $\PP$ is the blow-up of $\PP^N$ at $h$ points in general position and $\Lambda\subset\PP$ is a general linear subspace, then there is an isomorphisms $\mathcal{M}_\pp(k,a;\Lambda, {\bf m})\cong \ls_{N,a}(\Lambda,{\bf m}):=\mathcal{O}_{\PP^N}(a)\otimes\mathcal{I}_{\Lambda}\otimes\mathcal{I}_Z$.

Finally, since $\F$ is a $\PP^N$ blown-up at a point $p$ on the Veronese variety $V\subset \PP^N$ and since $\tilde{V}$ is the strict transform of $V$ via this blow-up, then
there is an isomorphisms $\mathcal{M}_\F(k,a;\tilde{V},{\bf m})\cong \ls_{N,k}(V,a):=\mathcal{O}_{\PP^N}(k)\otimes\mathcal{I}_{V}\otimes\mathcal{I}_{p^a}$.

\end{remark}


\subsection{Toric degeneration of the Veronese variety}\label{toric-section}

We refer to \cite{fulton-toric-book} for details on projective toric varieties associated with convex lattice polytopes and to \cite{GKZ} for details on coherent triangulations.
Let $\Delta_n\subset\mathbb{R}^n$  be the $n$-dimensional simplex obtained as the convex hull of the points $(0,\dots,0), (1,0,\dots,0),\dots,(0,\dots,0,1)$. Consider $d\Delta_n=\Delta_N+\cdots+\Delta_N$, where $+$ denotes the Minkowski sum of polytopes in $\mathbb{R}^n$.
The polytope $d\Delta_n$ defines the $d$-th Veronese embedding of $\PP^n$ in $\PP^N$, with $N=N_d={{N+d}\choose N}-1$, that we shall call  $V=V_{n,d}$ as in the previous section.
Consider the lattice $\mathbb{Z}^n\subset\mathbb{R}^n$ and the set of lattice points $\mathcal{A}:=d\Delta_n\cap \mathbb{Z}^n$. We have $\sharp\mathcal{A}=N+1$ and each such point corresponds to a coordinate point of the ambient space $\PP^N$.
\subsubsection{Degenerating the Veronese to a union of linear spaces}
Take a \emph{regular triangulation} of  $d\Delta_n$, that is a decomposition of $d\Delta_n$ into  a finite union of simplices
$$\bigcup_{i=1}^{d^n}S_i,$$
where 
\begin{itemize}
\item each $S_i$ is obtained as the convex hull of $n+1$ non-aligned points of $\mathcal{A}$,
\item $\sharp S_i\cap\mathbb{Z}^n=n+1$,
\item for $i\ne j$, $S_i\cap S_j$ is a common faces of $S_i$ and $S_j$ (possibly empty),
\item there is a strictly convex piecewise linear function $\lambda: \mathbb{R}^n\to\mathbb{R}$ whose domains of linearity are precisely the $S_i$'s.
\end{itemize}
We can always assume that $S_1$ is the convex hull of the lattice points $(0,\dots,0)$, $(1,0,\dots,0),\dots,(0,\dots,0,1)$, so that it lies at a corner of $d\Delta_n$.
Consider for example, for $n=2$ and $d=3$, the subdivision into nine triangles and the piecewise linear function inducing it, as shown in Figure \ref{regular subd}. 
\begin{figure}[h!]
\setlength{\unitlength}{0.5mm}
\begin{center}
\begin{picture}(45,50)(0,-10)
\put(0,0){\line(1,0){45}}
\put(0,15){\line(1,0){30}}
\put(0,30){\line(1,0){15}}
\put(0,0){\line(0,1){45}}
\put(15,0){\line(0,1){30}}
\put(30,0){\line(0,1){15}}
\put(0,45){\line(1,-1){45}}
\put(0,30){\line(1,-1){30}}
\put(0,15){\line(1,-1){15}}
\end{picture} \quad\quad 
\includegraphics[width=11em]{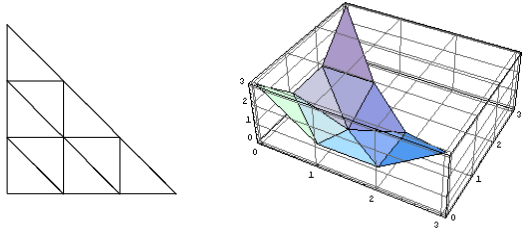}
\end{center}
\caption{A regular triangulation  of $3\Delta_2$.}\label{regular subd}
\end{figure}
In this figure, $S_1$ is the triangle with vertices $(0,0),(1,0),(0,1)$.

Each $S_i$ defines a $\PP^n$ as a toric variety, which we will call $\Pi_i$, for $i=1,\dots,d^n$.
Since a regular triangulation of $d\Delta_n$ induces a 1-parameter embedded degeneration of $V_{n,d}\subset \PP^N$ to the union of toric varieties described by the $S_i$'s (see \cite{CDM} and \cite{Pos-SecDeg} for details on the 2-dimensional case), we have a degeneration of the $n$-dimensional Veronese variety $V_{n,d}$ to a union of $d^n$ $n$-planes:
$$
{\bf\Pi}:=\bigcup_{i=1}^{d^n}\Pi_i.
$$
The intersection table of these planes is encoded in the combinatorial data described by the triangulation, that is: if $S_i\cap S_j$ is $r$-dimensional, then $\Pi_i\cap\Pi_j\cong \PP^r$, for $0\le r\le n-1$. \begin{remark}\label{sink}
Because of the choice of $S_1$ made, we will say that $\Pi_1$ is a \emph{sink}. In practice this means that it is possible to choose a hyperplane of $\PP^N$ that contains every $S_i$, $i>1$, but that does not contain $S_1$.
\end{remark}
Moreover, the union of planes ${\bf\Pi}\subset\p^N$ is a torus invariant subscheme.
In fact, consider the simplex $\Delta_N$, which defines $\PP^N$ as a toric variety, with an action of the algebraic torus $(\mathbb{C}^\ast)^N$. Each $r$-dimensional face of $\Delta_N$ corresponds to a torus invariant linear subspace of dimension $r$ of $\PP^N$. In particular vertices of $\Delta_N$ are in one-to-one correspondence with $N+1$ linearly general points of $\PP^N$, which we may assume to be the coordinate points, after  a suitable change of coordinates. Each $r$-dimensional face of $\Delta_N$ corresponds to a $\PP^r$ spanned by $r+1$ coordinate points. 
Since each $\Pi_i$ is the linear span of $n+1$ coordinate points of $\PP^N$, then the union ${\bf\Pi}$ is embedded in a copy of $\PP^N$ and it is invariant under the action of the torus $(\mathbb{C}^\ast)^N$. 
In particular, each $\Pi_i$ will correspond to a \emph{marked} $n$-dimensional face of $\Delta_N$ and we have $d^n$ such marked faces.

\subsubsection{Degenerating a linear system intepolating the Veronese}

We now consider the linear systems on $\PP^N$ of degree$-k$ hypersurfaces containing the Veronese variety on the one hand, and the union of $n-$planes ${\bf\Pi}$ on the other hand:
\begin{align*}
\ls_{N,k}(V_{n,d})&:=\mathcal{O}_{\PP^N}(k)\otimes\mathcal{I}_{V_{n,d}},\\
\ls_{N,k}({\bf\Pi})&:=\mathcal{O}_{\PP^N}(k)\otimes\mathcal{I}_{\bf\Pi}
\end{align*}
\begin{lemma}\label{deg-linsis-Vero-Pi}
In the above notation, we have 
 $$\dim\ls_{N,k}(V_{n,d})\le \dim\ls_{N,k}({\bf\Pi}).$$
\end{lemma}
\begin{proof}
Since ${\bf\Pi}$ is a flat degeneration of $V_{n,d}$, then the statement follows by semi-continuity of the function $\dim$.
\end{proof}
 

 \section{Some auxiliary linear systems}\label{section-linearsystems}
 
 \subsection{Hypersurfaces containing a linear subspace and $h$ double points in general position}

It is a well celebrated result of Alexander and Hirschowitz that if we impose $h$ double points in general position to the hypersurfaces of degree $d$ of $\PP^N$, there is only a finite list of cases where the dimension is larger than that obtained via a parameter count, that is
$$\edim\ls_{N,d}(2^h)=\max\left\{-1,{{N+d}\choose N}-h(N+1)-1\right\}\le\dim\ls_{N,d}(2^h),$$
where the expression on the left hand side is the \emph{expected dimension} of the linear system.
\begin{theorem}[Alexander-Hirschowitz Theorem]\label{AH-thm}
The linear system $\ls_{N,d}(2^h)$ has the expected dimension, except in the following cases:
\begin{itemize}
\item $d=2$ and $N\ge 2$, $2\le h\le N$;
\item $d=3$ and $(N,h)=(4,7)$;
\item $d=4$ and $(N,h)=(2,5),(3,9),(4,14)$.
\end{itemize}
\end{theorem}
The interested reader may see \cite{A},\cite{AH1},\cite{AH2},\cite{AH3},\cite{AH4}] for the original proof based on specialisation of points (Horace method), and \cite{BO} and \cite{Chandler} for a simplified proof. An alternative proof via a different degeneration construction can be found here \cite{PosPhD,Po}. This inspired the degeneration approach developed in Section \ref{degeneration-section} that will be used to prove the main result, Theorem \ref{main-thm1}.

In this section we want to present an analogous result about linear systems with $h$ imposed double points in general position and a linear subspace.
Let $\Lambda\subset\pp^N$ be a general linear subspace of dimension $n$ 
and let $Z\subset\PP^N$ be a double point scheme with support a set of points in general position. Let $ \mathcal{I}_\Lambda$ be the ideal sheaf of $\Lambda\subset\pp^N$ and let
$\mathcal{I}_{Z}$ be the ideal sheaf of $Z\subset\PP^N$. Consider
 the sheaf
$$\ls_{N,k}(\Lambda,2^h):= \mathcal{O}_{\pp^N}(k)\otimes \mathcal{I}_\Lambda\otimes\mathcal{I}_{Z}.$$
Since the Hilbert polynomial of $\Lambda\subset\pp^N$ at degree $k$ is 
$${{n+k}\choose n}$$ 
 or, in other terms, 
\begin{equation}\label{Hilb-Lambda}
\h^0(\mathcal{O}_{\pp^N}(k)\otimes\mathcal{I}_\Lambda)={{N+k}\choose N}-{{n+k}\choose n},
\end{equation}
and since the scheme given by $h$ double points in general position of $\PP^N$ has length $h(N+1)$,
 we can give the following definitions.
 \begin{definition}\label{exp dim Lambda and general double points}
The \emph{virtual  dimension} of the linear system $\ls_{N,k,\Lambda}(2^h)$ of hypersurfaces of $\PP^n$ that vanish along a linear subspace of dimension $n$, $\Lambda\subset\PP^N$, and double at $h$ points in general position is the following integer:
 $$
\vdim \ls_{N,k}(\Lambda,2^h)={{N+k}\choose N}-{{n+k}\choose n}- h(N+1)-1.
$$
The \emph{expected dimension} of $\ls_{N,k}(\Lambda,2^h)$ is
 $$
\edim \ls_{N,k}(\Lambda,2^h)=\max\left\{-1, \vdim \ls_{N,k}(\Lambda,2^h)\right\}.
$$
\end{definition}


Since $\Lambda$ and the scheme of double points are disjoint, the virtual dimension provides a lower bound to the actual dimension:
\begin{equation}\label{virt-lowerbound-V-double}
\edim \ls_{N,k}(\Lambda,2^h)\le \dim \ls_{N,k}(\Lambda,2^h).
\end{equation}

\begin{proposition}\label{dim-doublepoints-on-Lambda}
Let $\Lambda\subset\pp^N$ be a linear subspace of dimension $n$ and let $Z_{\Lambda}\subset\PP^N$ be a double point scheme supported on a collection of points in general position in $\PP^N$. Then if 
\begin{equation}\label{second bound on h}
h\le \frac{1}{N+1}{{N+k-1}\choose N},
\end{equation}
and $(N,k-1,h)$ is not in the list of exceptions of Theorem \ref{AH-thm}, and $k\ge2$, then 
\begin{equation}\label{dim Lambda and points}
\dim \ls_{N,k}(\Lambda,2^h)= \edim \ls_{N,k}(\Lambda,2^h).
\end{equation}
\end{proposition}
\begin{proof}
%

If $k=2$ and $h=0$, the conclusion follows from \eqref{Hilb-Lambda}. If $k=2$ and $h=1$, it is easy to see that all elements of $\LL_{N,2}(\Lambda, 2)$ are pointed quadric cones containing $\Lambda$ and hence we have the isomorphism $\LL_{N,2}(\Lambda, 2)\cong\LL_{N-1,2}(\Lambda)$. By \eqref{Hilb-Lambda}, we have that $\dim\LL_{N-1,2}(\Lambda)={{N+1}\choose2}-{{n+2}\choose2}$. We conclude noticing that the latter equals the expected dimension of $\LL_{N,2}(\Lambda, 2)$.
 
Now, assume $k\ge 4$ and consider the following exact sequence obtained by restricting $\ls_{N,k}(\Lambda,2^h)$ to a general hyperplane $H\subset\pp^N$ such that $\Lambda\subseteq H$:
\begin{equation}\label{Castelnuovo-sequence}
0\to  \ls_{N,k-1}(2^h)\to \ls_{N,k}(\Lambda,2^h)\to \ls_{N,k}(\Lambda,2^h)|_H\subseteq \ls_{N-1,k}(\Lambda).
\end{equation}
Under the assumption  \eqref{second bound on h} and using Theorem \ref{AH-thm}, the kernel system $ \ls_{N,k-1}(2^h)$ has dimension equal to its virtual dimension:
$$
\dim\ls_{N,k-1}(2^h)={{N+k-1}\choose N}-h(N+1)-1,
$$
and in particular $H^1(\PP^N,\dim\ls_{N,k-1}(2^h))=0$, so that we have the following exact sequence in cohomology:
$$
0\to  \H^0(\PP^N,\ls_{N,k-1}(2^h))\to H^0(\PP^N,\ls_{N,k}(\Lambda,2^h))\to \H^0(H, \ls_{N,k}(\Lambda,2^h)|_H)\to0.
$$
Moreover by \eqref{Hilb-Lambda}
$$
\h^0(\PP^{N-1},\ls_{N-1,k}(\Lambda))={{N-1+k}\choose N-1}-{{n+k}\choose n},
$$ 
and so 
$$
\h^0(H,\ls_{N,k}(\Lambda,2^h)|_H)\le{{N-1+k}\choose N-1}-{{n+k}\choose n}.
$$
From the exact sequence of global sections we obtain:
\begin{align*}
\h^0(\PP^N,\ls_{N,k}(\Lambda,2^h))&=\h^0(\PP^N,\ls_{N,k-1}(2^h))+\h^0(H, \ls_{N,k}(\Lambda,2^h)|_H)\\
& \le {{N+k}\choose N}-{{n+k}\choose n}- h(N+1).
\end{align*}
We conclude the proof of this case using \eqref{virt-lowerbound-V-double}.

Finally, assume that $k=3$. In this case, the bound on the number of points is $h\le \frac{N}{2}+1$. We consider the restriction to a general hyperplane containing $\Lambda$ as in \eqref{Castelnuovo-sequence}. The kernel system is special by Theorem \ref{AH-thm}, and one can easily check that it has dimension 
$$
\dim\ls_{N,2}(2^h)={{N+2}\choose2}-h(N+1)+{h\choose2}-1,
$$
see for instance \cite[Section 1.2.1]{PosPhD}. Moreover, as a simple consequence of B\'ezout's Theorem, the linear system $\ls_{N,3}(\Lambda,2^h)$ contains in its base locus the lines spanned by pairs of points, each of which intersects  $H$ in a point. We claim that the base locus of $\ls_{N,3}(\Lambda,2^h)$ is supported on the union of $\Lambda$ and these lines. This implies that the restricted system is the complete linear system of cubics containing $\Lambda$ and passing simply through the ${h\choose2}$ trace points:
$$
\ls_{N,3}(\Lambda,2^h)|_H= \ls_{N-1,3}(\Lambda,1^{{h\choose2}}).
$$
We claim that the linear system on the right hand side of the above expression is non-special, namely that the scheme given by $\Lambda$ and the simple points impose independent conditions to the cubics of $\PP^{N-1}$. 
This shows that 
\begin{align*}
\dim \ls_{N,3}(\Lambda,2^h)&=\dim\ls_{N,2}(2^h)+\dim \ls_{N-1,3}(\Lambda,1^{{h\choose2}})+1\\
&=\left({{N+2}\choose2}-h(N+1)+{h\choose2}-1\right)\\ &\quad +\left({{N+2}\choose3}-{{n+3}\choose3}-{h\choose2}-1\right)+1,
\end{align*}
which implies that $\ls_{N,3}(\Lambda,2^h)$ is non-special.

We are left to proving the two claims. 
For the second claim, first of all notice that there is a hyperplane inside $H$, containing all ${h\choose2}$ points. This can be taken to be the intersection with $H$ of a hyperplane of $\PP^N$ containing the $h$ original points. Call $H_1\subset H$ the intersection and restrict the linear system $\ls_{N-1,3}(\Lambda,1^{{h\choose2}})$ to it, giving rise to the following exact sequence:
$$
0\to  \ls_{N-1,2}(\Lambda)\to \ls_{N-1,3}(\Lambda,1^{{h\choose2}})\to \ls_{N-2,3}(1^{{h\choose2}}).
$$
Both external linear systems are non-special (for the restricted one see for instance \cite[Proposition 3.12]{Po}),  then so is the middle one, concluding the proof of the claim.

As for the first claim: we show that $\ls_{N,3}(\Lambda,2^h)$ has no additional base locus other than $\Lambda$ and the lines spanned by pairs of points. Let's call $p_1,\dots,p_h$ the $h$ assigned pints in general position. Assume that $q$ is a point in $\PP^N$ in linearly general position with respect to $p_1,\dots,p_h$. Since $h\le \frac{N}{2}+1<N$, there is a hyperplane $A$ containing $p_1,\dots,p_h$ but not containing $q$. Since $n<N$, there is a hyperplane $B$ containing $\Lambda$ but not containing $q$. The cubic $2A+B$ belongs to $\ls_{N,3}(\Lambda,2^h)$ proving that $q$ cannot be a base point. Assume now that $q$ is a point in $\PP^N$ not in linearly general position with respect to $p_1,\dots,p_h$, which means that there is a linear space spanned by some of the $p_i$'s containing $q$, but such that $q$ does not belong to any of the lines $\langle p_i,p_j\rangle$. Let $\langle p_i:i\in I_q\rangle$ be the minimum such linear span and choose two distinct indices $j_1,j_2\in I_q$. Let $A_1$ be a hyperplane containing all $p_i$'s with $i\ne j_1$ and let $A_2$ be a hyperplane containing all $p_i$'s with $i\ne j_2$.  Let $B$ be a hyperplane containing $\Lambda$, $p_{j_1}$ and $p_{j_2}$ and not containing $q$. The cubic $A_1+A_2+B$ belongs to $\ls_{N,3}(\Lambda,2^h)$ proving that $q$ cannot be a base point.
Finally, since the multiplicity of the general element of the linear system of cubics along  the line $q\in \langle p_i:i\in I_q\rangle$ is exactly $1$, then the above cases are exhaustive and this conclude the proof of the claim.
\end{proof}

\subsection{Hypersurfaces containing the Veronese variety and a fat point}

Let $V=V_{n,d}\subset\pp^n$ be the $d$-th Veronese embedding of $\pp^n$
and let $\{p^a\}\subset V\subset\PP^n$ be a fat point scheme with support on $V$. Let $ \mathcal{I}_V$ be the ideal sheaf of $V\subset\pp^N$ and let
$\mathcal{I}_{p^a}$ be the ideal sheaf of $Z\subset\PP^N$. Consider
 the sheaf
$$\ls_{N,k}(V,a):= \mathcal{O}_{\pp^N}(k)\otimes \mathcal{I}_V\otimes\mathcal{I}_{p^a}.$$
We are interested in computing the dimension of the space of global sections. 
The Hilbert polynomial of $V\subset\pp^N$ at degree $k$ is 
$${{n+kd}\choose n}$$ 
 or, equivalently,  we have 
$$
\dim\mathcal{O}_{\pp^N}(k)\otimes\mathcal{I}_V={{N+k}\choose N}-{{n+kd}\choose n}-1.
$$
The scheme given by a point of multiplicity $a$ of $\PP^N$ imposes 
\begin{equation}\label{fat-point-Hilb}
{{N+a-1}\choose N}
\end{equation} 
conditions to the hypersurfaces of $\PP^N$ of degree $k$.
Therefore the \emph{virtual dimension} of $\ls_{N,k}(V,a)$, obtained by a parameter count, is
$${{N+k}\choose N}-{{n+kd}\choose n}-{{N+a-1}\choose N}-1.$$
It does not yield a satisfactory notion of expected dimension for the linear system $\ls_{N,k}(V,a)$, due to the fact that the two subschemes $V$ and $\{p^a\}$ of $\PP^N$ have nonempty intersection so that some of the conditions imposed by them individually to the hypersurfaces of degree $k$ of $\PP^N$ will overlap. For instance, if we first impose $V$ and then $\{p\}$, clearly the latter will not give any independent condition, because, by the containment relation $p\in V$, $p$ is a base point of the linear system $\ls_{N,k}(V)=\mathcal{O}_{\pp^N}(k)\otimes \mathcal{I}_V$.
When the support of a fat point subscheme $Z=\{p^a\}\subset \PP^N$, whose length is given in \eqref{fat-point-Hilb}, lies on the $n$-dimensional subvariety $V$, the restriction $Z|_V\subset V$ is a subscheme of length $${{n+a-1}\choose n}.$$ 
Therefore we may define the following notion of expected dimension.

\subsubsection{A notion of expected dimension}
We introduce the following refined parameter count.
\begin{definition}\label{exp dim V and fat point}
Let $V\subset\pp^N$ be the $d$-th Veronese embedding of $\pp^N$ and let $\{p^a\}\subset\PP^N$ be a fat point scheme supported on $V$. The \emph{expected dimension} of $\ls_{N,k}(V, a)$, denoted by $\edim \ls_{N,k}(V,a)$, is the following integer:
$$
\max\left\{-1,{{N+k}\choose N}-{{n+kd}\choose n}- \left[{{N+a-1}\choose N}-{{n+a-1}\choose n}\right]-1\right\}.
$$
\end{definition}
That the integer of Definition \ref{exp dim V and fat point} is a lower bound to the actual dimension of $\ls_{N,k}(V,a)$ is not an obvious statement. We will show that it does  when $a\le k$.

\begin{proposition}\label{expdim-lowerbound}
Let $\nu_d: \pp^n\to\pp^N$ be the $d$-the Veronese embedding. Let $V=V_{n,d}:=\nu_d(\pp^n)\subset\pp^N$ and let $Z_{V}=\{p^a\}\subset\PP^n$ be an ordinary fat point of multiplicity $a\le k$  supported on $V$. Then
\begin{equation}\label{lower bound to dim V and points}
\dim \ls_{N,k}(V,a)\ge \edim \ls_{N,k}(V,a).
\end{equation}
\end{proposition}
\begin{proof}
We consider the linear system $\ls_{N,k}(a)=\mathcal{O}_{\PP^N}\otimes\mathcal{I}_{Z_V}$ of the degree-$k$ hypersurfaces of $\PP^N$ with a point of multiplicity $a$ with support on $V$. Restriction to $V$ gives the following Castelnuovo sequence: 
$$
0\to \ls_{N,k}(V,a)\to \ls_{N,k}(a)\to \ls_{N,k}(a)|_V.
$$
It is an easy observation that a fat point of multiplicity $a$  imposes independent conditions to the hypersurfaces of fixed degree  of $\PP^N$, as long as the multiplicity does not exceed the degree. Therefore we can obtain the dimension of the linear system $\ls_{N,k}(a)$ by a parameter count:
$$
\dim\ls_{N,k,V}(a)={{N+k}\choose N}-{{N+a-1}\choose N}-1
 $$
In particular $\h^1(\pp^N,\ls_{N,k}(a))=0$, so that we have the following sequence in cohomology:
\begin{align*}
0\to \H^0(\ls_{N,k}(V,a))\to \H^0(\ls_{N,k}(a))\to \H^0(\ls_{N,k}(a)|_V) 
\to \H^1(\ls_{N,k}(V,a))\to 0.
\end{align*}
Since the Veronese morphism $\nu_d:\pp^n\to \PP^N$ gives an isomorphism of $\PP^n$ to its image $V$, then the pull-back  of $\ls_{N,k}(a)|_V$ 
 is a linear system of degree-$kd$ hypersurfaces of $\pp^n$: $$
\nu_d^\ast (\ls_{N,k}(a)|_V)\subseteq \mathcal{O}_{\pp^n}(kd)\otimes\mathcal{I}_{Z'}=:\ls_{n,kd}(a)
$$
where $Z'\subset\pp^n$ is a fat point of multiplicity $a$ with support a general point of $V$.
Since $nk\ge a$ by the assumption, the linear system $\ls_{n,kd}(a)$  has dimension
$$
\dim\ls_{n,kd}(a)={{n+kd}\choose n}-{{n+a-1}\choose n}-1.
$$
From this we obtain 
$$
\dim\ls_{N,k}(a)|_V\le{{n+kd}\choose n}-{{n+a-1}\choose n}-1.
$$
Putting everything together gives the following inequalities
\begin{align*}
\h^0(\ls_{N,k}(V,a))&=\h^0(\ls_{N,k}(a))- \h^0(\ls_{N,k}(a)|_V)+\h^1(\ls_{N,k}(V,a))\\
&\ge\h^0(\ls_{N,k}(a))- \h^0(\ls_{N,k}(a)|_V)\\
&\ge \left({{N+k}\choose N}-{{N+a-1}\choose N}\right)- \left({{n+kd}\choose n}-{{n+a-1}\choose n}\right),
\end{align*}
which conclude the proof.
\end{proof}


\subsubsection{Dimensionality via apolarity and toric geometry}
Let $V=V_{n,d}\subset\PP^N$ be the Veronese variety and  let ${\bf\Pi}\subset\PP^N$ be a union of $n$-planes, degeneration of $V$, as in Section \ref{toric-section}. Let $p\in V$ and $p_0\in\Pi_1\subset{\bf\Pi}$ be general points.
Consider the linear systems on $\PP^N$
\begin{align*}
\ls_{N,k}(V,a)&:=\mathcal{O}_{\PP^N}(k)\otimes\mathcal{I}_V\otimes\mathcal{I}_{p^a},\\
\ls_{N,k}({\bf\Pi},a)&:=\mathcal{O}_{\PP^N}(k)\otimes\mathcal{I}_v\otimes\mathcal{I}_{p_0^a}.
\end{align*}
Building on  Lemma \ref{deg-linsis-Vero-Pi}, we obtain the following result.
\begin{proposition}
In the above notation, we have 
 $$\dim\ls_{N,k}(V,a)\le \dim\ls_{N,k}({\bf\Pi},a).$$
\end{proposition}
\begin{proof}
Since $p$ is a general point on $V$, we may assume that it degenerates to a general point $p_0\in S_1$.
Since ${\bf\Pi}\cup\{p_0^a\}$ is a flat degeneration of the scheme $V\cup\{p^a\}$, then the Hilbert function of the former is at most that of latter, by semi-continuity. This concludes the proof.
\end{proof}


\begin{proposition}\label{proposition-toric}
In the above notation and for any $1\le a\le k$, then
$$\dim\ls_{N,k}({\bf\Pi},a)={{N+k}\choose N}-{{n+kd}\choose n}-{{N+a-1}\choose N}+{{n+a-1}\choose n}-1.$$
\end{proposition}
\begin{proof}
Given the union ${\bf\Pi}:=\bigcup_{i=1}^{d^n}\Pi_i$ of torus invariant $n$-planes of $\PP^n$, with $\Pi_1$ a sink and  $p_0$ supported generically on $\Pi_1$,  there is a torus invariant hyperplane $H$ such that $\Pi_1\cap H$ is an $(n-1)$-plane and $\Pi_i\subset H$ for $2\le i\le d^n$ (cf. Remark \ref{sink}).
We can always assume that $p_0$ is a coordinate point of $\PP^N$ and we can call
$p_1,\dots,p_N$ the other coordinate (torus invariant) points of $\PP^N$. Hence we can choose, without loss of generality, that $\Pi_1=\langle p_0,\dots, p_n\rangle$ and $H=\langle p_1,\dots,p_N\rangle$, so that $p_0\in\Pi_1$ and $p_i\notin\Pi_i$ for $i\ge 2$.

Let $R=\mathbb{C}[x_0,\dots,x_N]$ be the homogeneous polynomial ring of $\PP^N$ and consider the ideals $\mathcal{I}_{p_0}\subset \mathbb{C}[x_0,\dots,x_N]$ and $\mathcal{I}_{\Pi_i}\subset \mathbb{C}[x_0,\dots,x_N]$ 
\begin{align*}
\mathcal{I}_{p_0}&=\langle x_1,\dots, n_N\rangle,\\
\mathcal{I}_{\Pi_1}&=\langle x_{n+1},\dots, n_N\rangle,\\
\mathcal{I}_{\Pi_i}&=\langle x_{i_{n+1}},\dots, n_{i_N}\rangle,\ i\ge 2,\\
\mathcal{I}_{H}&=\langle x_0\rangle.
\end{align*}
By construction,  for $i\ge 2$ , we have $0\in \{{i_{n+1}},\dots, {i_N}\}$.
Using Notation \ref{apolarity-notation}, we compute:
\begin{align*}
\left[\mathcal{I}_{p_0^a}^{-1}\right]_k&=\{y_0^{k-l}F_l(y_1,\dots,y_N): F_l\in S_l, 0\le l\le a-1\},\\
\left[\mathcal{I}_{\Pi_1}^{-1}\right]_k&=\{F_k(y_0,\dots,y_n): F_k\in S_k\},\\
\left[\mathcal{I}_{\Pi_i}^{-1}\right]_k&=\{F_k(y_{i_0},\dots,y_{i_n}): F_k\in S_k\}, i\ge 2,
 \end{align*}
 where for $i\ge 2$, the index set $\{i_0,\dots,i_n\}$ is the complement of $\{i_{n+1},\dots, {i_N}\}\subset\{0,\dots,N\}$.
We have the following intersections 
\begin{align*}
\left[\mathcal{I}_{p_0^a}^{-1}\right]_k\cap \left[\mathcal{I}_{\Pi_i}^{-1}\right]_k&=\emptyset, i\ge 2,\\
\left[\mathcal{I}_{p_0^a}^{-1}\right]_k\cap\left[\mathcal{I}_{\Pi_1}^{-1}\right]_k&=\{y_0^{k-l}F_l(y_1,\dots,y_n): F_l\in S_l\}.
 \end{align*}
 We compute the dimension of the latter intersection:
 \begin{align*}
 \dim \left[\mathcal{I}_{p_0^a}^{-1}\right]_k\cap\left[\mathcal{I}_{\Pi_1}^{-1}\right]_k&=\dim \{y_0^{k-l}F_l(y_1,\dots,y_n): F_l\in S_l, 0\le l\le a-1\}\\
 &=\sum_{l=0}^{a-1}\dim \{F_l(y_1,\dots,y_n): F_l\in S_l\}\\
 &=\sum_{l=0}^{a-1}{{n-1+l}\choose{n-1}}\\
 &={{n+a-1}\choose{n}},
 \end{align*}
 where the last equality follows a standard relation of Newton coefficients,  commonly known as the \emph{hockey stick identity}.
 
 The number of conditions imposed to the linear system of degree-$k$ hypersurfaces of $\PP^N$ by the scheme $\{p_0^a\}\cup{\bf\Pi}$ is  the dimension of the linear span of $\left[\mathcal{I}_{p_0^a}^{-1}\right]_k$ and $\left[\mathcal{I}_{\Pi_i}^{-1}\right]_k$, for $i=1\dots,d^n$, which is the following integer:
\begin{align*}
&\quad \dim \left[\mathcal{I}_{p_0^a}^{-1}\right]_k+\dim\left\langle \left[\mathcal{I}_{\Pi_i}^{-1}\right]_k, i=1\dots,d^n\right\rangle-\dim \left[\mathcal{I}_{q_0^a}^{-1}\right]_k\cap\left[\mathcal{I}_{\Pi_1}^{-1}\right]_k\\
&= {{N+a-1}\choose{N}}+{{n+kd}\choose{n}}-{{n+a-1}\choose{n}}.
\end{align*}
 
\end{proof}

\begin{corollary}\label{corollary-toric}
The linear system $\dim\ls_{N,k}(V,a)$ has the expected dimension according to Definition \ref{exp dim V and fat point}.
\end{corollary}
\begin{proof}
It follows from Propositions \ref{expdim-lowerbound} and  \ref{proposition-toric}.
\end{proof}

\section{Proof of the main theorem}\label{section-proof}

We are ready to prove our main theorem, Theorem \ref{main-thm1}.
Thanks to Proposition \ref{identifiability tool proposition}, computing the dimension of the $h$-secant varieties of the $(d,k)$-Veronese variety $V^k_{n,d}\subset \p^{N_{dk}}$ is equivalent to computing the dimension of the linear system in $\p^N$ of hypersurfaces containing the  Veronese variety and double at $h$ general points. 

Let $V_{n,d}\subset\pp^N$ be the $d$-th Veronese embedding of $\pp^n$
and let $Z\subset\PP^n$ be a double point scheme with support a set of points in general position. Let $ \mathcal{I}_V$ be the ideal sheaf of $V\subset\pp^N$ and let
$\mathcal{I}_{Z}$ be the ideal sheaf of $Z\subset\PP^N$. Consider
 the sheaf
$$\ls_{N,k}(V,2^h):= \mathcal{O}_{\pp^N}(k)\otimes \mathcal{I}_V\otimes\mathcal{I}_{Z}.$$
%
Since the Hilbert polynomial of $V\subset\pp^N$ in degree $k$ is 
${{n+kd}\choose n}$
 or, in other terms, 
\begin{equation}\label{HilbFnVeronese}
\dim\mathcal{O}_{\pp^N}(k)\otimes\mathcal{I}_V={{N+k}\choose N}-{{n+kd}\choose n}-1,
\end{equation}
and since $h$ double points in general position of $\PP^N$ impose $h(N+1)$ conditions to the hypersurfaces of $\PP^N$ of degree $k$,
 we can give the following definitions.
 \begin{definition}\label{exp dim V and general double points}
The \emph{virtual  dimension} of the linear system $\ls_{N,k,V}(2^h)$ of hypersurfaces of $\PP^n$ that vanish along the Veronese variety $V=V_{n,d}\subset\pp^N$ and double at $h$ points in general position is the following integer: 
$$
\vdim \ls_{N,k}(V,2^h)={{N+k}\choose N}-{{n+kd}\choose n}- h(N+1)-1.
$$
The \emph{expected dimension}  is
 $$
\edim \ls_{N,k}(V,2^h)=\max\left\{-1, \vdim \ls_{N,k}(V,2^h)\right\}.
$$
\end{definition}
Since $V$ and the scheme of double points are disjoint, the virtual dimension provides a lower bound to the actual dimension:
\begin{equation}\label{virtual-lowerbound-Veronese-double}
\dim \ls_{N,k}(V,2^h)\ge \edim \ls_{N,k}(V,2^h).
\end{equation}

Using a degeneration argument, we shall show that if the number of points $h$ is not too large, then the linear system $\ls_{N,k,V}(V,2^h)$ has dimension equal to the expected dimension.
\begin{theorem}\label{main-thm-linearsystems}
Let $\nu_d: \pp^n\to\pp^N$ be the $d$-the Veronese embedding. Let $V=V_{n,d}:=\nu_d(\pp^n)\subset\pp^N$ and let $Z_{V}\subset\PP^n$ be a double point scheme supported on $h$ points in general position of $\PP^N$. Then if $k\ge 3$ and
\begin{equation}\label{main-bound}
h\le \frac{1}{N+1}{{N+k-3}\choose N}
\end{equation}
then
\begin{equation}\label{lower bound to dim V and points}
\dim \ls_{N,k}(V,2^h)= \edim \ls_{N,k}(V,2^h).
\end{equation}
\end{theorem}
\begin{proof}
Using  \eqref{virtual-lowerbound-Veronese-double}, it is enough to prove that the inequality $\dim \ls_{N,k}(V,2^h)\leq \edim \ls_{N,k}(V,2^h)$ holds.

If $k=3$, then $h=0$ so the statement follows from equation \eqref{HilbFnVeronese}.
For $k\ge 4$, we will prove the statement by means of the $\mathbb{FP}-$degeneration introduced in Section \ref{degeneratingVeronese-section},  applied to the line bundle 
\begin{align*}
\mathcal{L}_{\widetilde{\mathcal{X}}}:=&\mathcal{M}_{\widetilde{\mathcal{X}}}(k,k-1;\widetilde{\mathcal{V}},2,\dots,2).\\
\intertext{By Remark \ref{abuse-notation-ls}, the line bundle on the general fiber is isomorphic to }
\ls_t:=&\ls_{N,k}(V,2^h),\\
\intertext{while on the central fiber the linear systems  on the two components are the following:} 
\ls_\pp:=&\LL_{N,k-1}(\Lambda,2^h),\\
\ls_\F:=&\LL_{N,k}(V,k-1).\\
\intertext{We consider the restriction to $R=\pp\cap\F$: the kernels on the two components are, respectively:}
\hat{\ls}_\pp:=&\LL_{N,k-2}(\Lambda,2^h),\\
\hat{\ls}_\F:=&\LL_{N,k}(V,k).\\
\intertext{Since $R\cong\pp^{N-1}$, the two restricted systems satisfy the following:}
\mathcal{R}_\pp:=&\ls_\pp|_R\subset\ls_{N-1,k-1}(\Lambda_R),\\
\mathcal{R}_\F:=&\ls_\F|_R\subset\ls_{N-1,k-1}(\Lambda_R), 
\end{align*}
where we recall that  $\Lambda_R=\Lambda\cap R\cong\pp^{n-1}$.

We first look at the exceptional component $\pp$. By Proposition \ref{dim-doublepoints-on-Lambda}, since $k-2\ge 2$ and
$$
h\le\frac{1}{N+1}{{N+k-3}\choose{N}},
$$
both linear systems $\ls_\PP$ and $\hat{\mathcal{L}}_\pp$ have the expected dimension, that is 
\begin{align}
\dim\ls_\PP&={{N+k-1}\choose N}-{{n+k-1}\choose n}-h(N+1)-1\nonumber\\
\dim\hat{\mathcal{L}}_\pp&={{N+k-2}\choose N}-{{n+k-2}\choose n}-h(N+1)-1.\label{dim-hatLP}
\end{align}
Moreover, we have a short exact sequence of spaces of global sections:
$$
0\to \H^0(\pp,\hat{\mathcal{L}}_\pp)\to \H^0(\pp,{\mathcal{L}}_\pp)\to \H^0(R,\mathcal{R}_\pp)\to0.
$$
In particular, we can compute
\begin{align*}
\dim \mathcal{R}_\pp&=\dim{\mathcal{L}}_\pp-\dim \hat{\mathcal{L}}_\pp+1\\
&={{N+k-2}\choose {N-1}}-{{n+k-2}\choose {n-1}}-1\\
&=\dim \ls_{N-1,k-1}(\Lambda_R).
\end{align*}
We conclude that $\mathcal{R}_\pp$ is the complete linear system $$\mathcal{R}_\pp=\ls_{N-1,k-1}(\Lambda_R).$$

On the component $\F$, using Corollary \ref{corollary-toric},  we have that both $\ls_\F$ and $\hat{\ls}_\F$ have the expected dimension, that is
\begin{align}
\dim\mathcal{L}_\F&={{N+k}\choose N}-{{n+kd}\choose n}-\left({{N+k-2}\choose N}-{{n+k-2}\choose n}\right)-1\nonumber\\
\dim\hat{\mathcal{L}}_\F&={{N+k}\choose N}-{{n+kd}\choose n}-\left({{N+k-1}\choose N}-{{n+k-1}\choose n}\right)-1.\label{dim-hatLF}
\end{align}

We claim that 
$$\mathcal{R}_\F=\ls_{N-1,k-1}(\Lambda_R),$$
so that, together with the above argument, we have
\begin{equation}\label{L_R}
\mathcal{R}_\pp\cap\mathcal{R}_\F=\mathcal{O}_{\pp^{N-1}}(k-1)\otimes\mathcal{I}_{\Lambda_R}.
\end{equation}
In order to prove the claim, we observe that by semicontinuity, and precisely Formula \eqref{semicontinuity}, and by \eqref{virtual-lowerbound-Veronese-double}, we have 
\begin{equation}\label{chain-inequalities}
\dim\ls_0\ge \dim\ls_t\ge\edim\ls_t={{N+k}\choose N}-{{n+kd}\choose n}- h(N+1)-1.
\end{equation}
Using Formula \eqref{H0 central fiber}, i.e., 
\begin{align}\label{matching-theorem}
\dim\ls_0&=\dim\hat{\ls}_{\pp}+\dim\hat{\ls}_{\F}+\mathcal{R}_\pp\cap\mathcal{R}_\F
+2\end{align}
and observing that $\mathcal{R}_\pp\cap\mathcal{R}_\F=\mathcal{R}_\F$, 
we obtain
\begin{align*}
\dim\mathcal{R}_\F&\ge \edim\ls_t-\dim\hat{\ls}_\pp-\dim\hat{\ls}_\F-2={{N+k-2}\choose {N-1}}-{{n+k-2}\choose {n-1}}-1;
\end{align*}
the proof of the latter equality is easy and left to the reader. Since 
\begin{align*}
\dim\mathcal{R}_\F\le \dim\mathcal{O}_{\pp^{N-1}}(k-1)\otimes\mathcal{I}_{\Lambda_R}={{N+k-2}\choose {N-1}}-{{n+k-2}\choose {n-1}}-1,
\end{align*}
the claim follows and we have 
\begin{equation}\label{dim-R}
\dim\mathcal{R}_\PP\cap\mathcal{R}_\F ={{N+k-2}\choose {N-1}}-{{n+k-2}\choose {n-1}}-1,
\end{equation}
Using  \eqref{dim-hatLP},  \eqref{dim-hatLF}, \eqref{matching-theorem} and \eqref{dim-R}, we obtain 
$\dim\ls_0=\edim\ls_t$.  We conclude using \eqref{chain-inequalities}.

%
\end{proof}
Theorem \ref{main-thm1} is now just a corollary of what we just proved.

\begin{corollary}\label{cor-main-them1}
For $k\ge3$, if \eqref{main-bound} holds, then the $(d,k)-$Veronese variety $V_{n,d}^k\subset\p^{N_{dk}}$ is non-defective.
\end{corollary}
\begin{proof} It follows from Theorem \ref{main-thm-linearsystems} and Proposition \ref{identifiability tool proposition}.
\end{proof}

  Theorem \ref{main-thm2} is an easy consequence of Theorem \ref{main-thm1} and Theorem \ref{Mass-Mella}.
\begin{corollary}
For $k\ge3$, if $$h \leq \min\left\{\frac{1}{N+1} {N+k-3 \choose N}-1, \left\lfloor \frac{{{n+kd} \choose {n}}}{N+1} \right\rfloor-1\right\}$$ then the $(d,k)-$Veronese variety $V_{n,d}^k\subset\p^{N_{dk}}$ is $h$-identifiable.
\end{corollary}
\begin{proof}
It follows from Corollary \ref{cor-main-them1} and Theorem \ref{Mass-Mella}.
\end{proof}

\begin{remark}\label{generic}
Note that the bound given in Theorem \ref{main-thm1} can be strictly larger then the expected generic rank for the $(d,k)-$Veronese $V_{n,d}^k$, see for example the computations in Section \ref{section-bound} below. In this case Theorem \ref{main-thm1} proves the non defectivity of $V_{n,d}^k$ up to the generic rank. The identifiability property instead is more subtle and in order to be able to use Theorem \ref{Mass-Mella} we must ensure that the embedded secant variety  $\Sec_h(V_{n,d}^k)$ has the same dimension as the abstract secant variety $sec_h(V_{n,d}^k)$, see definition \ref{abstract}. Equivalently we require that the $(h+1)-$secant map is generically finite. Because of this assumption we are not able to prove generic identifiability and we must restrict ourselves to the range given in Theorem \ref{main-thm2}.

\end{remark}

\section{Asymptotical Bound}\label{section-bound}
In this section we relate our bound in Theorem \ref{main-thm1} with the one given in \cite{Nen}.
For the sake of completeness, we recall Nenashev's result first.
\begin{theorem}{\cite[Theorem 1]{Nen}}\label{thm:FroExp}
	Let $I$ be a homogeneous ideal generated by $h\in \mathbb{N}\setminus \{0\}$  
	generic elements of some nonempty variety $ \mathcal{D} \subseteq \Sym^r(\C^n) $ of $ r $-forms that is closed under linear transformations. Fix an integer $s\geq 0$. If $$h\leq \left({{r+s+n-1}\choose {n-1}}/{{s+n-1}\choose {n-1}}\right) -{{s+n-1}\choose {n-1}},$$  
	then the dimension of $I$ in degree $(r+s)$ is maximal, i.e. it equals $h {{s+n-1}\choose {n-1}}$.  
\end{theorem}

Note that when $r=d(k-1)$ and $s=r$ the component of $I$ of degree $r+s=dk$ gives exactly the dimension of the $h-$ secant variety $\Sec_h(V_{n,d}^k)$, where $\mathcal{D}$ is the tangential variety of $V_{n,d}^k$, i.e. $$\mathcal{D}=\{F^{k-1}G|F,G \in \C[x_0,\dots,x_n]_d\}.$$
We have the following consequence.

\begin{corollary}\label{cor:FroExp}
The dimension of $\Sec_h(V_{n,d}^k)$ is the expected one, i.e. $$\dim \Sec_h(V_{n,d}^k)=h  {{n+d}\choose {d}}-1$$ for $h \leq \frac{{{n+dk}\choose {dk}}}{{{n+d}\choose {d}}}-{{n+d}\choose {d}}$.
\end{corollary}

Note that if we fix $k,n$ and let $d \gg 0$ we have that $$\frac{{{n+dk}\choose {dk}}}{{{n+d}\choose {d}}}-{{n+d}\choose {d}} \sim k^n-d^n$$ and if $d \gg k$ the bound is trivial.
In Theorem \ref{main-thm1}  and under the same assumptions we get $$\frac{1}{N+1}{{N+k-3}\choose {N}} \sim d^{n(k-4)}$$ which gives non trivial bounds for $d \gg k$ when $k >4$.

\begin{center}
\includegraphics[scale=0.4]{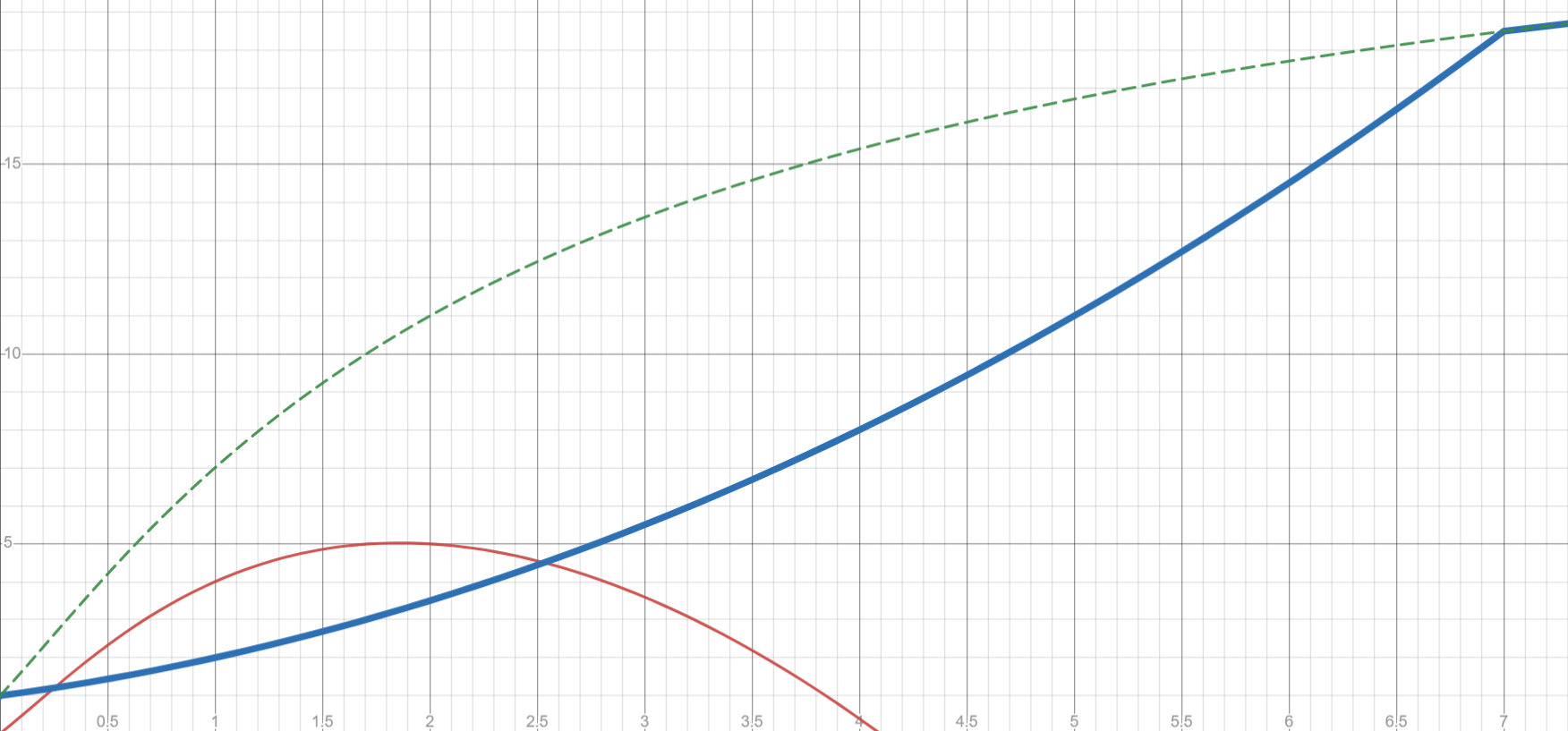}
\end{center}

The figure shows in \textcolor{red}{red} the bound given by Nenashev, in \textcolor{blue}{blue} the bound of Theorem \ref{main-thm2} as a function of $d$ and in \textcolor{green}{green} the expected generic rank ${{n+kd} \choose {kd}}/{{n+d} \choose {d}}$ of $V_{n,d}^k$. In this case we have set the values $k=5$ and $n=2$. For $d>3$ our bound overtakes Nenashev's and it continues to give informations also in the range $d>4$.  

We also notice that for $d$ big enough our bound in Theorem \ref{main-thm1} exceeds the expected generic rank ${{n+kd} \choose {kd}}/{{n+d} \choose {d}}$ of $V_{n,d}^k$, as predicted in Remark \ref{generic}. This shows that under the hypothesis of a very high degree reembedding of $\p^n$ the variety $V_{n,d}^k$ is never defective for every $k \geq 3$. 

\end{document}